\documentclass[10pt,onecolumn,aps,prd,floats,floatfix,superscriptaddress,nofootinbib,showkeys]{revtex4}
\usepackage{amsmath,amsfonts,dsfont,amsthm,amssymb}
\usepackage{caption}%or \usepackage[format=hang]{caption}
\usepackage{mathtools}
\usepackage{multirow}
\usepackage{graphicx}
\usepackage{subfig}
\usepackage{color}
\usepackage{epsfig}
\usepackage{verbatim}
\usepackage{hyperref}
\usepackage{bm}% New package added for producing bold greek letters e.g. \bm{beta}, be careful
\theoremstyle{plain}
\newtheorem{theorem}{Theorem}[section]

\theoremstyle{definition}
\newtheorem{definition}{Definition}[section]
\newtheorem{example}[subsubsection]{Example}
\theoremstyle{remark}

\numberwithin{equation}{section}

\begin{document}
%\title{Quasi-Harmonic Toric B{\'e}zier Surfaces}
\title{Quasi-Harmonic Constraints for Toric B{\'e}zier Surfaces }
\author{Daud Ahmad}
\thanks{daud.math@pu.edu.pk}
\thanks{One of the authors Daud Ahmad acknowledges the financial support of University of the Punjab, Research Project Grant 2015-16 No. D/999/Est.1 .}
\affiliation {Department of Mathematics, University of the Punjab, Lahore, Pakistan}
\author{Saba Naeem}
\thanks{saba.baloch289@gmail.com}
\affiliation {Department of Mathematics, University of the Punjab, Lahore, Pakistan}

\begin{abstract}
Toric B\'{e}zier patches generalize the classical tensor-product  triangular and rectangular B{\'e}zier surfaces, extensively used in $CAGD$. The construction of toric B\'{e}zier surfaces corresponding to multi-sided convex hulls for known boundary mass-points with integer coordinates (in particular for trapezoidal and hexagonal convex hulls) is given. For these toric B\'{e}zier surfaces, we find approximate minimal surfaces obtained by extremizing the quasi-harmonic energy functional. We call these approximate minimal surfaces as the quasi-harmonic toric B\'{e}zier surfaces. This is achieved  by imposing  the vanishing condition of gradient of the quasi-harmonic functional and obtaining  a set of linear constraints on  the unknown inner mass-points of the toric B\'{e}zier patch for the above mentioned convex hull domains, under which they are quasi-harmonic toric B\'{e}zier patches. This gives us the solution of the \textit{Plateau toric B\'{e}zier problem} for these illustrative instances for known convex hull domains.
\end{abstract}
%\pacs{}
\maketitle
\section{Introduction}\label{intro}
The theory of minimal surfaces  has its roots in the optimization problems of calculus of variations, based on the famous Euler- Lagrange equation
%\cite{gelfand2000calculus}
which is a second order partial differential equation (\emph{pde}).  The solution of  the  Euler-Lagrange equation targets to find a function that extremizes a given functional and has many  applications in the optimization theory. Many mathematicians have contributed to the subject of  optimization theory  and it has become a widely accepted discipline of Mathematics and Physics.  A minimal surface is a surface which locally minimizes its area or equivalently a surface whose mean curvature vanishes everywhere on the surface. In the similar context, a  problem known as the Plateau problem \cite{Osserman1986,Nitsche1989} consists of finding the surface with least surface area bounded by a given boundary curve. It is named after Belgian physicist Joseph. A. Plateau~\cite{Plateau}
who experimentally demonstrated in 1849 that minimal surfaces can be associated to the soap films spanned by wire frames of different shapes. In the meantime, many mathematicians developed their interest in finding a minimal surface spanned by a fixed boundary curve such as Schwarz \cite{Schwarz} (who studied the triply periodic surfaces namely the CLP (crossed layers of parallels), D (diamond), P (primitive), H (hexagonal) and  T (tetragonal) surfaces, Weierstrass \cite{Osserman1986}, Riemann \cite{Osserman1986}  and R. Garnier \cite{Garnier} in the late 19th century. However, these were minimal surfaces for particular boundaries, until in 1931, American mathematician J. Douglas \cite{Douglas} and in 1933 Hungarian Tibor Rad\'o \cite{Rado} independently proved the existence of a minimal surface spanned by a closed curve by replacing the area functional by rather a simpler integral, now known as the Douglas-Dirichlet functional. The Douglus-Dirichlet functional does not have square root in its integrand as is the case with the area functional which makes it a suitable choice  as an alternative to the area functional.

Exact mathematical solutions are known only for some specific boundaries. It is possible to find numerically the solution of a wide variety of problems giving rise to approximate minimal surfaces. Coppin and Greenspan \cite{COPPIN1988315} used a computer model of molecular structure and forces to approximate a minimal surface. K. Koohestani \cite{Koohestani20142071} also suggested the method involving non-linear force density to find minimal surfaces for membrane structures. Brakke \cite{brakke1992} used the finite element method to approximate parameterized minimal surfaces.  Level set method was proposed by Chopp \cite{Chopp199377} to cope with topological variations of a surface under linear convergence, whereas a variational approach to minimize the area of triply periodic surfaces was proposed by Jung \emph{et al.} \cite{Jung:2007}. R{\o}nquiust and Tr{\aa}sdahl \cite{Trasdahl20114795} introduced an iterative scheme which involves parameterization of higher order polynomials to achieve a numerical approximation of a minimal surface with fixed boundaries.  Similarly, Li \emph{et al.} \cite{Li20136415} numerically approximated the minimal surfaces with geodesic constraints over boundary curves. Kassabov \cite{Kassabov2014441} derived an equation of a canonical parameterized minimal surface and also pointed out its application. Xu \emph{et al.} \cite{Xu:2015:EFP:2778896.2779237} proposed a parametric form of polynomial minimal surface with varying degrees which posses interesting properties helpful for geometric modeling in CAD.

Alternative energy functionals for minimization may  be used to find an approximate minimal surface of a certain restricted class of surfaces. One of the widely used restriction is to find a minimal  B\'{e}zier surface among all the B\'{e}zier surfaces
\begin{equation}\label{beziersurface}
 \mathbf{x}(u,v)=\sum\limits_{i=0}^{n}{\sum\limits_{j=0}^{m}{B_{i,j}^{n,m}\left( u,v \right)}}\text{ }{{\mathbf{P}}_{ij}},
\end{equation}
with
\begin{equation}\label{classicalbernpolys1}
B_{i,j}^{n,m}\left( u,v \right)=B_{i}^{n}\left( u \right)B_{j}^{m}\left( v \right),
\end{equation}
spanned by a given boundary in which  $\mathbf{P}_{ij}$ represents a two dimensional control net over the domain $D = [0,1] \times [0,1]$ with $u,v$ as the surface parameters, the bivariate functions $\{B_{i,j}^{n,m}\left( u,v \right):{{\mathbf{R}}^{2}}\to \mathbf{R}\}$
are the blending functions to specify the shape of the surface and
\begin{equation}\label{classicalbernpolys2}
B_{i}^{n}\left( u \right) = \binom {n}{i}{{u}^{i}}{{\left( 1-u \right)}^{n-i}}
\end{equation}
are  the Bernstein polynomials of degree $n$ with $\binom {n} {i}= \frac{n!}{i!\left( n-1 \right)!}$ as the binomial coefficients.

An extremal of discrete version of Dirichlet functional giving minimal B\'{e}zier surfaces can be seen in the Monterde work \cite{Monterde2004}. X. D. Chen, G. Xu, and Y. Wang. \cite{Chen2009} found  approximate minimal surfaces as the solution of Plateau-B\'{e}zier problem using extended Dirichlet functional and the extended bending energy functional, the surfaces depend on the parameters $\lambda$ and $\alpha$ (as they appear in eqs. (4) and (5) of the ref.\cite{Chen2009}) for simple estimates of these parameters. Hao \emph{et al.} ~\cite{Hao2012} investigated the Plateau-quasi-B\'{e}zier problem, minimizing thereby the Dirichlet functional of surfaces for  more generalized  borders including the boundary curves like polynomial curves,  catenaries and circular arcs. Another restriction could be to find a  parametric polynomial minimal surface as has been proposed by Xu and Wang~\cite{Xu2010} to obtain a minimal surface for quintic parametric polynomial surface having the prescribed borders as polynomial curves. Ahmad and Masud \cite{Ahmad201472,dabm2013,Ahmad20151242} gave an algorithm to find a quasi-minimal surface, variationally improving the non-minimal initial surface spanned by a fixed boundary composed of finite number of curves by minimizing its $rms$ mean curvature functional instead of area functional which involves a square root in its integrand and applied this technique to a variety of surfaces.
The idea may be extended to more generalized surfaces called toric B\'{e}zier surfaces to obtain a quasi-minimal surface by minimizing the quasi-harmonic functional as is done by Xu \emph{et al.} \cite{Xu2015} to find the quasi-harmonic surface as the solution of Plateau-B{\'e}zier problem. The related class of surfaces is called the harmonic mapping. The harmonic mappings find significant importance in the literature of minimal surfaces for the isothermal parameterization of the  surfaces \cite{Osserman1986,Chern1955}. This means that a positive definite metric in  two dimensions
\begin{equation}
  d{{s}^{2}}=E\left( x,y \right)d{{x}^{2}}+2F\left( x,y \right)dxdy+G\left( x,y \right)d{{y}^{2}},
\end{equation}
defined in the neighbourhood of a surface $\mathbf{x}(x,y)$ in local coordinates $\left(x,y\right)$  takes the form
\begin{equation}
  d{{s}^{2}}=\lambda^2 \left( x,y \right)\left( d{{x}^{2}}+d{{y}^{2}} \right),
\end{equation}
 ($i.e.$ $E\left( x,y \right)=G\left( x,y \right)={{\lambda }^{2}}\left( x,y \right),F\left( x,y \right)=0$)  in the isothermal coordinates $\left(x,y\right)$.
If a surface is parameterized using the isothermal parameterization ~\cite{Chern1955}, then such a parameterization is minimal if the coordinate functions are harmonic. In other words, a surface with isothermal parameterization is a minimal surface if and only if it is a harmonic surface. This is also useful in finding a minimal surface associated to a class of surfaces namely the B{\'e}zier surfaces. Monterde and Ugail~\cite{MonterdeUgail2004}  indicated that harmonic B{\'e}zier surfaces can only be specified by opposite boundary control points and thus making it impracticable to generate a harmonic B{\'e}zier surface from the prescribed four boundary B{\'e}zier curves. In order to overcome this difficulty, Xu \emph{et al.}~\cite{Xu2015} proposed the quasi-harmonic surfaces which serve as the solution surfaces for Plateau-B{\'e}zier problem. They also showed  that in particular cases when the corners of B{\'e}zier surface are almost isothermal, quasi-harmonic surfaces are better approximations when compared to surfaces generated by Dirichlet method.

Polynomial functions and splines are widely used in many structural design program softwares.   The fundamental units of modeling a surface geometrically are the classical B{\'e}zier triangles and rectangular tensor product patches \cite{FARIN198683} in computer aided geometric designing $\left(CAGD\right)$, however some applications require a more generalized form of multi-sided $C^{\infty}$ patches rather than the classical B{\'e}zier surfaces. J. Warren~\cite{Warren1994} realized the usage of real toric surfaces in $CAGD$. His notable contribution is construction of a hexagonal patch from a rational B{\'e}zier triangle with zero weights and the corresponding control points located appropriately. The multi-sided patches bear more flexibility and present interesting mathematical structures when dealt through  Krasauskas's toric B{\'e}zier patches \cite{ref1}.  Toric  B{\'e}zier patches are the generalization of  the classical B{\'e}zier patches that deal only with triangular or rectangular patches. In 2002, Krasauskas and Goldman \cite{KrausGoldman} presented the construction of toric B{\'e}zier patches of depth $d$ by using the de Casteljau pyramid algorithm and blossoming algorithm for the associated patches. In recent work by Gang Xu, Tsz-Ho Kwok and Charlie C.L. Wang \cite{XU20171}, a B-spline volumetric parameterization is constructed with semantic features for isogeometric analysis.

Further developments in toric B\'{e}zier surfaces include the work of  Garc\'ia-Puente \emph{et al.} \cite{Puente:2011}, they illustrated the geometrical importance of the structural system of toric B{\'e}zier patches, Sun and Zhu \cite{Sun2015255,Sun01062016} discussed the $G^{1}$ continuity of toric B{\'e}zier surfaces and found approximate minimal toric B\'{e}zier surfaces by minimizing the Dirichlet functional.

In this paper, we construct quasi-harmonic toric B\'{e}zier patches defined over multi-sided convex hulls with prescribed boundary mass-points by extremizing the quasi-harmonic functional to generate a system of linear equations for the unknown inner mass-points. This enables us to write down the parametric form of the solution of the Plateau-toric B\'{e}zier problem. The paper is organized as follows: In section \ref{ToricBezier}, we give the preliminary introduction to toric B{\'e}zier patch of depth $d$ in general and its construction consisting of indexing lattice polygon domains and the associated toric Bernstein polynomials. In the following sections \ref{qfunctional} and \ref{conditions}, we utilize the  quasi-harmonic energy functional as the objective functional to obtain the necessary and sufficient conditions for a toric B{\'e}zier patch to be a quasi-harmonic toric B{\'e}zier patch which serves as the solution to the Plateau-toric B\'{e}zier problem. Finally, in section \ref{examples}, we construct quasi-harmonic toric B\'{e}zier patches defined over trapezoidal convex hulls and hexagonal convex hull as illustrative applications. Constraints on mass-points of the toric B\'{e}zier patches defined over the above mentioned multi-sided domains are obtained by solving the respective systems of linear equations for the inner unknown mass-points.
For the prescribed boundary mass points, quasi-harmonic toric B\'{e}zier patches, as illustrative applications,  have also been obtained and shown that  the inner mass-points satisfy the computed constraints.
\section{Toric B{\'e}zier Patches and Related Terminology}\label{ToricBezier}
In computer aided geometric designing $\left(CAGD\right)$, three and four-sided patches namely the triangular and rectangular B{\'e}zier patches are commonly used for surface modeling but a multi-sided generalization of these B{\'e}zier schemes is required in order to fill $n$-sided holes.
One of such schemes used to define multi-sided $C^{\infty}$ patches is the Krasauskas's \textit{Toric B{\'e}zier patch} as introduced in \cite{ref1}. A scheme in section~\ref{conditions} is given to obtain quasi-harmonic toric B\'{e}zier surface by extremizing the quasi-harmonic functional  introduced in the section~\ref{qfunctional}.
To comprehend the construction of these toric B{\'e}zier patches and then to extremize a given functional to find an approximate minimal surface, we give below the related terminology for the reader to get familiar with  lattice polygons, Bernstein basis functions for these polygons,  discrete convolution indexed by Minkowski sum and finally the construction of toric B\'{e}zier patches for given depth $d$.
\begin{definition}(\textbf{Lattice Polygons}) The polygon formed by connecting the outer most sequence of points in the finite set $\sigma\in \mathbb{Z}^2$ in the plane is called the \textit{lattice polygon}. The finite set $\sigma$ is used as the index set for control points $\{P_{\sigma_i}\}_{\sigma_i\in\sigma}$ to form a polygonal array of control points.
\end{definition}
\noindent The lattice polygons for the classical tensor-product B{\'e}zier patch and triangular B{\'e}zier patch are lattice rectangle and lattice triangle respectively which form the array of their corresponding control points. Other examples of multi-sides lattice polygons are given in fig~\ref{latticepolygons}.
\begin{figure}[tbh!]
\begin{center}
\includegraphics[width=40mm]{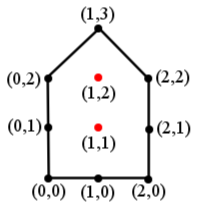}
\hspace{1cm}
\includegraphics[width=50mm]{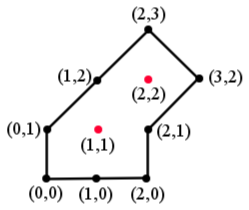}
\end{center}
\caption{Multi-sided lattice polygons, a lattice pentagon (left) and a lattice hexagon (right) with inner lattice points (red dots).}
\label{latticepolygons}
\end{figure}

\begin{definition}(\textbf{Bernstein Polynomial Functions for Lattice Polygons})\label{defn:bpfflp}
Let $\sigma=\{\sigma_1,\sigma_2,...,\sigma_m\}\in \mathds{Z}^2$ be the set of finite integers in $uv$-plane. The lattice polygon $I_\sigma$ denotes  the convex hull of $\sigma$ with corner points $v_1,v_2,...,v_n$ and $\overline{L}_k(u,v)= \alpha_ku+\beta_kv+\gamma_k,\, k=1,2,..n$, the $kth$ edge of the convex hull $I_\sigma$. In addition, the direction of the normal vector $(\alpha_k,\beta_k)$ to the line ${{\bar{L}}_{k}}\left( u,v \right)$ is in the convex hull $I_\sigma$ and $(\alpha_k,\beta_k)$ is the shortest normal vector with integer coordinates in that direction.
\end{definition}
%lote different commands for bold , few work but others do not work 1. $\mathbf{\beta}$, 2.  $\boldmath{\beta}$, 3. \boldmath$\beta$,  4. $\pmb{\beta}$, 5. $\boldsymbol{\beta}$ 6. $\bm{\beta}$\\
 The  Bernstein polynomials    ${\bm{\beta}_{\sigma_i}(u,v)}_{\sigma_i\in\sigma}$  for $\left(u,v\right)$  in the convex hull $I_\sigma$, for toric B\'{e}zier patch can be written as
\begin{equation}\label{bernteinpolynomiallatticepolygon}
{\bm{\beta}_{\sigma_i}(u,v)}=c_{\sigma_i} \{\overline{L}_1(u,v)\}^{\overline{L}_1(\sigma_i)}\{\overline{L}_2(u,v)\}^{\overline{L}_2(\sigma_i)}...\{\overline{L}_n(u,v)\}^{\overline{L}_n(\sigma_i)},
\end{equation}
where positive arbitrary normalizing constants $c_{\sigma_i}$ are  the coefficients of basis functions,  chosen appropriately to get certain desired formulas.  For toric B{\'e}zier patches, the Bernstein polynomials for lattice polygon $\{\bm{\beta}_{\sigma_i}(u,v)\}_{(\sigma_i)\in I_\sigma}$ have the analogous properties as that of classical Bernstein polynomials \eqref{classicalbernpolys2}) for which the classical bivariate functions (eq.~\eqref{classicalbernpolys1} $\{B^{n,m}_{i,j}\}(u,v)$ are
\begin{equation}\label{bernpolynomial}
    B_{i,j}^{n,m}\left( u,v \right)=\binom{n}{i}\binom{m}{j} u^i \left(1-u\right)^{n-i}v^j \left(1-v\right)^{m-j},
\end{equation}
(for $i\in\{0,...,n\},j\in\{0,...,m\}$) used to construct  the triangular or rectangular B{\'e}zier patches. These  Bernstein polynomials $\{\bm{\beta}_{\sigma_i}(u,v)\}_{(\sigma_i)\in I_\sigma}$ (eq.~\eqref{bernteinpolynomiallatticepolygon}) indexed by the set $\sigma$ with  lattice polygon $I_\sigma$ having  corner points $v_1,v_2,...,v_n\,$   satisfy the following properties:  1) $\bm{\beta}_{\sigma_i}(u,v)>0$ inside the lattice polygon $I_\sigma$, 2) $\bm{\beta}_{\sigma_i}(u,v)=0$ on the edge $v_k v_{k+1}$, if and only if $\sigma_i\notin v_k v_{k+1}$, 3) $\bm{\beta}_{\sigma_i}(u,v)=1$ if $\sigma_i=v_k$ and 4) $\{\bm{\beta}_{\sigma_i}(u,v)\}$ are polynomial functions.
\begin{definition}(\textbf{Toric B{\'e}zier Patch})
A toric B{\'e}zier patch is a rational surface $\mathcal{P}(u,v)$ in the real projective space $\mathbb{RP}^4$ of dimension $4$    with control structure consisting of  mass-points $\{(\omega_{\sigma_i} P_{\sigma_i},\omega_{\sigma_i})\}$ indexed by the lattice polygon $I_\sigma$. The  mass-points $\{(\omega_{\sigma_i} P_{\sigma_i},\omega_{\sigma_i})\}$  are four dimensional elements with $\omega_{\sigma_i}$ as  the scaler weights corresponding  to control points $P_{\sigma_i}$ in space. The Bernstein polynomials for lattice polygon $\bm{\beta}_{\sigma_i}(u,v)$ as given in eq.~\eqref{bernteinpolynomiallatticepolygon} are  the blending functions which serve as the basis functions for toric B\'{e}zier patches defined over the domain lattice polygon $I_\sigma$  and they are chosen to obtain the desired shape of the surface.
 \noindent The  toric B{\'e}zier surface $\mathcal P(u,v)$ is defined by the expression
\begin{equation}\label{toricbezierpatch}
\mathcal{P}(u,v)=\sum\limits_{\sigma_i\in I_{\sigma}}\bm{\beta}_{\sigma_i}(u,v)\left(\omega_{\sigma_i} P_{\sigma_i},\omega_{\sigma_i}\right),~~~(u,v)\in I_\sigma,
\end{equation}
\end{definition}
\noindent where Bernstein polynomials $\{\bm{\beta}_{\sigma_i}(u,v)\}_{(\sigma_i)\in I_\sigma}$ are given in eq.~\eqref{bernteinpolynomiallatticepolygon}.
\noindent A rational surface may be obtained by dividing the surface eq.~\eqref{toricbezierpatch} by $\sum\limits_{\sigma_i\in I_{\sigma}}\bm{\beta}_{\sigma_i}(u,v)$ provided that $\sum\limits_{\sigma_i\in I_{\sigma}}\bm{\beta}_{\sigma_i}(u,v)\neq0$, throughout the domain.
Krasauskas and Goldman \cite{KrausGoldman} introduced the concept of depth for toric B{\'e}zier patches which is the analogue of degree used to define the classical higher order B{\'e}zier surfaces. It is based on the depth of lattice polygons defined with the help of repeated Minkowski sums.
\begin{definition}(\textbf{Minkowksi sum})\label{minkowskisum}
Let $A$ and $B$ be any two sets of $p$-tuples. The Minkowski sum $A\oplus B$ of these two sets is the set with the sum of all elements from $A$ and all elements of $B$ given by,
\begin{equation*}
A\oplus B=\{a+b | a\in A, b\in B\}.
\end{equation*}
\end{definition}
\begin{definition}(\textbf{Discrete convolution indexed by Minkowski sum})\label{discreteconvolution}
Let $P=\{P_a| a\in A\}$ and $Q=\{Q_b| b\in B\}$ be two arrays. Then the discrete convolution $P\otimes Q$ indexed by the Minkowksi sum $A\oplus B$ i.e., $P\otimes Q=\{(P\otimes Q)_c | c\in A\oplus B\}$ is defined as
\begin{equation*}
(P\otimes Q)_c=\sum\limits_{a+b=c} P_a Q_b.
\end{equation*}
\end{definition}
\noindent The indexing of discrete convolution indexed by Minkowski sum may be used to define toric B\'{e}zier patches with depth $d$, as given below. The depth $d$ of toric B{\'e}zier patches as expressed by Krasasuskas and Goldman \cite{KrausGoldman} is the analogue of degree used to define the classical higher order B{\'e}zier surfaces. It is based on the depth of lattice polygons defined with the help of repeated Minkowski sums as given above (definitions  \ref{minkowskisum} and \ref{discreteconvolution}).
\begin{definition}(\textbf{Toric B{\'e}zier Patch with depth $d$})
 Let $\sigma^{d}=\overbrace{\sigma \oplus \sigma ...\oplus \sigma}^{d\text{-}fold}$ be the $d$-fold Minkowski sum of $\sigma$ and $I^d$, the corresponding convex hull of $\sigma^{d}$. Then the toric Bernstein basis functions $\{\bm{\beta}^{d}_{\gamma}(u,v)\}_{\gamma\in\sigma^d}$ on $I^d$ are given by convolution of the Bernstein basis function $\{\bm{\beta}_{\sigma}(u,v)={\bm{\beta}_{\sigma_i}(u,v)}\}_{\sigma_i\in\sigma}$ indexed by $\sigma^{d}$, $i.e.$,
\begin{equation}
\{\bm{\beta}^{d}_{\gamma}(u,v)\}_{\gamma\in\sigma^d}=\overbrace{\bm{\beta}_{\sigma}(u,v)\otimes \bm{\beta}_{\sigma}(u,v)\otimes...\otimes\bm{\beta}_{\sigma}(u,v)}^{d\text{-}fold}.
\end{equation}
A toric B{\'e}zier patch defined on lattice polygon of depth $d$ and the corresponding convex hull $I^d$ of $\sigma^{d}$ in the projective space is a surface parameterized by the map $\mathcal{P}: I^d\rightarrow \mathbb{RP}^{4}$ (for $\left(u,v\right)\in I^d$) is defined as,
\begin{equation}\label{polynomialpatch}
\mathcal{P}(u,v)=\sum\limits_{\gamma\in \sigma^d} \bm{\beta}^{d}_{\gamma}(u,v) \left(\omega_\gamma p_\gamma, \omega_\gamma\right),
\end{equation}
 the control structure consists of the mass-points $\left\{\left(\omega_\gamma p_\gamma, \omega_\gamma\right)\right\}_{\gamma\in\sigma^{d}}$, where $\{p_\gamma\}_{\gamma\in\sigma^{d}}$ are the control points and $\{\omega_\gamma\geq 0\}_{\gamma\in\sigma^{d}}$ are the respective weights. ${\mathcal{P}^{d}_{\gamma}(u,v)}_{\gamma\in\sigma^{d}}$ are the blending functions, known as the \textit{toric Bernstein basis functions} for $I^d$.
\end{definition}
The toric B{\'e}zier patches are the rational surfaces lying in the affine or projective spaces. The derivative of a rational surface is not that straightforward in general but rather a little complicated. It is however advantageous  to find  the derivatives of the numerator and denominator parts of the rational surface first and then  to apply  the quotient rule of derivation to get  the derivative of the quotient. Therefore, instead of derivative of the rational toric B{\'e}zier patch, the derivative of the corresponding toric B{\'e}zier surface in the space of mass-points is more useful.
A detailed account of finding derivative of toric B\'{e}zier patch of depth $d$ $w.r.t.$ the surface parameters $u$ and $v$ can be seen in \cite{KrausGoldman}  (pages 82-84).
The partial derivative $w.r.t.$ $u$ of  Bernstein polynomials $\bm{\beta}^{d}_{\gamma}(u,v)$ for lattice polygons of depth $d$ is given by the following expression
\begin{equation}\label{toric1stderivative}
    \frac{\partial \bm{\beta}^{d}_{\gamma}(u,v)}{\partial u}=d \sum\limits_{\sigma_i\in\sigma}\frac{\partial \bm{\beta}_{\sigma_i}(u,v)}{\partial u}\bm{\beta}^{d-1}_{\gamma-\sigma_{i}}(u,v),
\end{equation}
which leads to the first order partial differentiation $w.r.t.$ $u$ of  the polynomial patch $\mathcal{P}(u,v)$ eq.~\eqref{polynomialpatch} and is given by
\begin{equation}
\mathcal{P}_{u}(u,v)=d\sum\limits_{\gamma\in\sigma^d}\left(\sum\limits_{\sigma_i\in\sigma}\frac{\partial \bm{\beta}_{\sigma_i}(u,v)}{\partial u}\bm{\beta}^{d-1}_{\gamma-\sigma_{i}}(u,v)\right)\left(\omega_\gamma p_\gamma, \omega_\gamma\right).
\end{equation}
\noindent The second order partial derivatives of toric B\'{e}zier patch $w.r.t.$ its parameters $u$ and $v$ (later to be used in next section) can be computed and they are
\begin{equation}\label{derivativeuu}
\begin{split}
\mathcal{P}_{uu}(u,v)=&d\sum\limits_{\gamma\in\sigma^d}\left(\sum\limits_{\sigma_i\in\sigma}\frac{\partial^{2} \bm{\beta}_{\sigma_i}(u,v)}{\partial u^2}\bm{\beta}^{d-1}_{\gamma-\sigma_{i}}(u,v)\right)\left(\omega_\gamma p_\gamma, \omega_\gamma\right)\\
+&d(d-1)\sum\limits_{\gamma\in\sigma^d}\left(\sum\limits_{\sigma_i\in\sigma}\frac{\partial \bm{\beta}_{\sigma_i}(u,v)}{\partial u}\left(\sum\limits_{\delta_i\in\sigma,\delta_i\neq\sigma_i}\frac{\partial \bm{\beta}_{\delta_i}(u,v)}{\partial u}\bm{\beta}^{d-2}_{\gamma-\sigma_{i}-\delta_{i}}(u,v)\right)\right)\left(\omega_\gamma p_\gamma, \omega_\gamma\right),
\end{split}
\end{equation}
\begin{equation}\label{derivativevv}
\begin{split}
\mathcal{P}_{vv}(u,v)=&d\sum\limits_{\gamma\in\sigma^d}\left(\sum\limits_{\sigma_i\in\sigma}\frac{\partial^{2} \bm{\beta}_{\sigma_i}(u,v)}{\partial v^2}\bm{\beta}^{d-1}_{\gamma-\sigma_{i}}(u,v)\right)\left(\omega_\gamma p_\gamma, \omega_\gamma\right)\\
+&d(d-1)\sum\limits_{\gamma\in\sigma^d}\left(\sum\limits_{\sigma_i\in\sigma}\frac{\partial \bm{\beta}_{\sigma_i}(u,v)}{\partial v}\left(\sum\limits_{\delta_i\in\sigma,\delta_i\neq\sigma_i}\frac{\partial \bm{\beta}_{\delta_i}(u,v)}{\partial v}\bm{\beta}^{d-2}_{\gamma-\sigma_{i}-\delta_{i}}(u,v)\right)\right)\left(\omega_\gamma p_\gamma, \omega_\gamma\right).
\end{split}
\end{equation}
Above partial derivatives of $\mathcal{P}(u,v)$ are helpful in the extremization of the quasi-harmonic functional used as objective function to obtain quasi-harmonic toric B\'{e}zier patch as the solution of Plateau toric B\'{e}zier problem, the task accomplished  in section~\ref {conditions}. The next section gives a brief description of the energy functionals that can be used as objective functions for extremization purpose to obtain an approximate minimal surface.
\section{Quasi-Harmonic Functional}\label{qfunctional}
To find an approximate minimal surface, several energy functionals have been used instead of area functional itself which involves a square root in its integrand. These functionals may be  extremized to obtain quasi-minimal surfaces with prescribed boundary in general. Following section gives a brief description of different energy functionals which may be used as objective functions to trigger the extremization process for different surfaces along with the quasi-harmonic functional that is used in our next section to obtain a quasi-harmonic B\'{e}zier patch as an  approximate solution to the Plateau-toric B\'{e}zier problem.
In an optimization problem, one needs to  minimize the area functional (eq.~\eqref{areafunctional} for any surface $\mathbf{x}(u,v)$. The area functional of the toric B{\'e}zier surface $\mathcal{P}(u,v)$ is
\begin{equation} \label{areafunctional}
\mathcal{A}(\mathcal{P}(u,v))=\int\limits_{I^d} |\mathcal{P}(u,v)_{u}\times\mathcal{P}(u,v)_{v}| du dv,
\end{equation}
where $ I^d \subset \mathbb{Z}^{2} $ is the parametric domain over which the surface $ \mathcal{P}(u,v) $ is defined as a map and $\mathcal P_{u}(u,v)$ and $\mathcal {P}_{v}(u,v)$ are the partial derivatives of $ \mathcal{P}(u,v) $ with respect to parameters $u$ and $v$. However, the non-linearity of this functional makes it difficult to find the  solution of Plateau problem in general.  Douglus \cite{Douglas} replaced the area functional for a surface $\mathbf{x}(u,v)$ with a relatively easy to manage Dirichlet functional
\begin{equation}
D\left(\mathbf{x}(u,v)\right)=\frac{1}{2}\int\limits_{R}\left(\|\mathbf{x}_u\|^2+\|\mathbf{x}_v\|^2\right)du dv.
\end{equation}
This functional was utilized by Monterde \cite{Monterde2004} to solve the Plateau-B\'{e}zier problem. Sun and Zhu \cite{Sun01062016} found the extremals of toric B\'{e}zier surfaces by minimizing the Dirichlet functional.

Monterde and Ugail \cite{Monterde2006}, in 2006,  introduced a general biquadratic functional
\begin{equation}
\mathcal{L}\left(\mathbf{x}(u,v)\right)=\frac{1}{2}\int\limits_{R}\left(a\|\mathbf{x}_{uu}\|^2+b\langle\mathbf{x}_{uu},\mathbf{x}_{uv}\rangle+c|\mathbf{x}_{uv}\|^2+d\langle\mathbf{x}_{uv},\mathbf{x}_{vv}\rangle+e\|\mathbf{x}_{vv}\|^2\right) du dv,
\end{equation}
with $a,b,c,d$ and $e$ being the real constants. By assigning different values to these constants, the functional could be reduced to other alternative functionals used for minimizing purposes such as  Farin and Hansford functional \cite{farin}, standard biharmonic functional introduced by Schneider and Kobbelt \cite{Schneider2001359} or Bloor and Wilson's modified biharmonic functional \cite{Bloor2005203}.
 \begin{comment}
 Its customary to deal with minimal surfaces as harmonic isothermal
surfaces \cite{Fomenko1991,Chern1955}. A regular surface $\mathbf{x}=\mathbf{x}(u,\text{ }v)$ in parametric form is
 isothermal if $E= {{\mathbf{x}}_{u}} \cdot {{\mathbf{x}}_{u}} = {{\mathbf{x}}_{v}} \cdot {{\mathbf{x}}_{v}}
 =G = \chi^2(u,v) \text{ and } F= {{\mathbf{x}}_{u}} \cdot {{\mathbf{x}}_{v}}
 =0$. The following is the fundamental
theorem which justifies this practice.
\begin{theorem}
Let $\mathbf{x}\left( u,v \right)=\left( x_1\left( u,v \right),\text{
}x_2\left( u,v \right), x_3\left( u,v \right) \right)$  be a parameterized
surface and assume that $\mathbf{x}$ is isothermal. Then $\mathbf{x}$ is a
minimal surface if and only if its coordinate functions
$x_1(u,v),x_2(u,v),x_3(u,v)$ are harmonic.
\end{theorem}
 \end{comment}
 The solution of the area problem for B\'{e}zier patches by extremizing the quasi-harmonic functional
\begin{equation}\label{}
\mathcal{H}\left(\mathbf{x}(u,v)\right)=\int\limits_R\left(\mathbf{x}_{uu}+\mathbf{x}_{vv}\right)^2du dv.
\end{equation}
for the surface $\mathbf{x}(u,v)$ is already known \cite{Xu2015}. We choose this quasi-harmonic functional as an objective function to find the solution of Plateau's toric B\'{e}zier problem, as mentioned earlier that the toric B\'{e}zier patches generalize the classical rational triangular and tensor-product B{\'e}zier surfaces defined over multi-sided domains.
It  gives \cite{Xu2015} better approximation of surfaces with lesser area and smaller mean curvature values at arbitrary points when compared to the Dirichlet functional  for B\'{e}zier surfaces. The quasi-harmonic functional, taken as an objective function, for the toric B\'{e}zier patch $\mathcal{P}(u,v)$ (eq.~\eqref{toricbezierpatch}) is given by
\begin{equation}\label{quasi-harmonicfunctional}
\mathcal{H}(\mathcal{P}(u,v))=\int\limits_{I^d} \left(\mathcal{P}_{uu}(u,v)+\mathcal{P}_{vv}(u,v)\right)^{2} dudv,
\end{equation}
where $\mathcal{P}_{uu}(u,v)$ and $\mathcal{P}_{vv}(u,v)$ are given by eqs.~\eqref{derivativeuu} and ~\eqref{derivativevv}.
\noindent In the following section, necessary and sufficient condition for a toric B\'{e}zier patch to be a quasi-harmonic toric B\'{e}zier is computed by extremizing the above mentioned quasi-harmonic functional eq.~\eqref{quasi-harmonicfunctional}.
\section{Quasi-harmonic Toric B{\'e}zier patches for a given boundary}\label{conditions}
For the Plateau Toric B{\'e}zier problem, we minimize the quasi-harmonic functional to get a quasi-harmonic toric B{\'e}zier patch $\mathcal{P}(u,v)$. For this, we find the gradient of the $\mathcal{H}(\mathcal{P}(u,v))$ with respect to the inner unknown mass-points $\left(\omega_\lambda p_\lambda, \omega_\lambda\right)$ and equate it to zero  to find the constraints as linear equations under which the $\mathcal{P}(u,v)$ is quasi-harmonic toric B\'{e}zier patch.
\begin{theorem}
If the mass-points associated to the boundary lattice points of the convex hull $I^d$ of the toric B{\'e}zier patch $\mathcal{P}(u,v)=\sum\limits_{\gamma\in \sigma^d} \bm{\beta}^{d}_{\gamma}(u,v) \left(\omega_\gamma p_\gamma, \omega_\gamma\right)$ are given, the patch $\mathcal{P}(u,v)$ is quasi-harmonic toric B{\'e}zier surface if and only if the inner unknown mass-points $\left(\omega_\lambda p_\lambda, \omega_\lambda\right)$ associated to the lattice points of the convex hull satisfy the following system of  linear equations:
\begin{equation}\label{finalfunctional}
\int\limits_{I^d}\sum\limits_{\gamma\in\sigma^d}
\left((\xi^{\lambda,u}+(d-1)\eta^{\lambda,u})+(\xi^{\lambda,v}+(d-1)\eta^{\lambda,v})\right)
\left((\xi^{\gamma,u}+(d-1)\eta^{\gamma,u})+(\xi^{\gamma,v}+(d-1)\eta^{\gamma,v})\right)(\omega_\gamma p_\gamma, \omega_\gamma) dudv =0,
\end{equation}
where the coefficients $\xi^{\gamma,u}$ and  $\eta^{\gamma,u}$ are,
\begin{equation}\label{xietacoefficients}
\begin{aligned}
&\xi^{\gamma,u}=\sum\limits_{\sigma_i\in\sigma}\frac{\partial^{2} \beta_{\sigma_i}(u,v)}{\partial u^2}\beta^{d-1}_{\gamma-\sigma_{i}}(u,v),\\
&\eta^{\gamma,u}=\sum\limits_{\sigma_i\in\sigma}\frac{\partial \beta_{\sigma_i}(u,v)}{\partial u}\left(\sum\limits_{\delta_i\in\sigma,\delta_i\neq\sigma_i}\frac{\partial \beta_{\delta_i}(u,v)}{\partial u}\beta^{d-2}_{\gamma-\sigma_{i}-\delta_{i}}(u,v)\right).
\end{aligned}
\end{equation}
Other coefficients $\xi^{\gamma,v}$, $\xi^{\lambda,u}$, $\xi^{\lambda,v}$, $\eta^{\gamma,v}$, $\eta^{\lambda,u}$ and $\eta^{\lambda,v}$ are obtained by replacing $u$ by $v$ and  $\gamma$ by  $\lambda$ in above eq.~\eqref{xietacoefficients}.
\end{theorem}
\begin{proof}
The quasi-harmonic functional~\eqref{quasi-harmonicfunctional} can be rewritten as
\begin{equation}\label{quasi-harmonicfunctional1}
\mathcal{H}(\mathcal{P}(u,v))=\int\limits_{I^d} \left\langle \mathcal{P}_{uu}(u,v),\mathcal{P}_{uu}(u,v)\right\rangle+\left\langle \mathcal{P}_{vv}(u,v),\mathcal{P}_{vv}(u,v)\right\rangle+2\left\langle \mathcal{P}_{uu}(u,v),\mathcal{P}_{vv}(u,v)\right\rangle dudv,
\end{equation}
where the operator $\left\langle ,\right\rangle$ denotes the inner product of the two functions. For an inner mass point $\left(\omega_\lambda p_\lambda, \omega_\lambda\right)$, $\lambda\in\sigma^d$ and $a\in\{1,2,3,4\}$, the gradient of the quasi-harmonic functional with respect to the coordinates of $\left(\omega_\lambda p_\lambda, \omega_\lambda\right)$ is given by
\begin{equation}\label{functional}
\begin{split}
\frac{\partial\mathcal{H}(\mathcal{P}(u,v))}{\partial\left(\omega_\lambda p_\lambda, \omega_\lambda\right)^{a}}=
2\int\limits_{I^d} &\left\langle\frac{\partial \mathcal{P}_{uu}(u,v)}{\partial\left(\omega_\lambda p_\lambda, \omega_\lambda\right)^{a}},\mathcal{P}_{uu}(u,v)\right\rangle +\left\langle\frac{\partial \mathcal{P}_{vv}(u,v)}{\partial\left(\omega_\lambda p_\lambda, \omega_\lambda\right)^{a}},\mathcal{P}_{vv}(u,v)\right\rangle\\ +&\left\langle \frac{\partial \mathcal{P}_{uu}(u,v)}{\partial\left(\omega_\lambda p_\lambda, \omega_\lambda\right)^{a}},\mathcal{P}_{vv}(u,v)\right\rangle +\left\langle \frac{\partial \mathcal{P}_{vv}(u,v)}{\partial\left(\omega_\lambda p_\lambda, \omega_\lambda\right)^{a}},\mathcal{P}_{uu}(u,v)\right\rangle dudv.
\end{split}
\end{equation}
Differentiating partially $\mathcal{P}_{uu}(u,v)$ and $\mathcal{P}_{vv}(u,v)$, the $2^{nd}$ order partial derivatives (eqs.~\eqref{derivativeuu} and \eqref{derivativevv} respectively) of the toric B\'{e}zier patch $\mathcal P(u,v)$  $w.r.t.$ the inner mass-point coordinates $\left(\omega_\lambda p_\lambda, \omega_\lambda\right)$ gives us
\begin{equation}\label{derivativeuucntrlpoint}
\begin{split}
&\frac{\partial \mathcal{P}_{uu}(u,v)}{\partial\left(\omega_\lambda p_\lambda, \omega_\lambda\right)^{a}}\\&
=d\left(\sum\limits_{\sigma_i\in\sigma}\frac{\partial^{2} \bm{\beta}_{\sigma_i}(u,v)}{\partial u^2}\bm{\beta}^{d-1}_{\lambda-\sigma_{i}}(u,v)\right)e^a
+d(d-1)\left(\sum\limits_{\sigma_i\in\sigma}\frac{\partial \bm{\beta}_{\sigma_i}(u,v)}{\partial u}\left(\sum\limits_{\delta_i\in\sigma,\delta_i\neq\sigma_i}\frac{\partial \bm{\beta}_{\delta_i}(u,v)}{\partial u}\bm{\beta}^{d-2}_{\lambda-\sigma_{i}-\delta_{i}}(u,v)\right)\right)e^a,
\end{split}
\end{equation}
and
\begin{equation}\label{derivativevvcntrlpoint}
\begin{split}
&\frac{\partial \mathcal{P}_{vv}(u,v)}{\partial\left(\omega_\lambda p_\lambda, \omega_\lambda\right)^{a}}\\
&=d\left(\sum\limits_{\sigma_i\in\sigma}\frac{\partial^{2} \bm{\beta}_{\sigma_i}(u,v)}{\partial v^2}\bm{\beta}^{d-1}_{\lambda-\sigma_{i}}(u,v)\right)e^a
+d(d-1)\left(\sum\limits_{\sigma_i\in\sigma}\frac{\partial \bm{\beta}_{\sigma_i}(u,v)}{\partial v}\left(\sum\limits_{\delta_i\in\sigma,\delta_i\neq\sigma_i}\frac{\partial \bm{\beta}_{\delta_i}(u,v)}{\partial v}\bm{\beta}^{d-2}_{\lambda-\sigma_{i}-\delta_{i}}(u,v)\right)\right)e^a.
\end{split}
\end{equation}
It is to be noted that in above eqs.~\eqref{derivativeuucntrlpoint} and \eqref{derivativevvcntrlpoint}, the coefficients $\bm{\beta}^{d-2}_{\lambda-\sigma_i-\delta_i}(u,v)=0$ if $\lambda-\sigma_i-\delta_i\notin \sigma^{d-2}$, $e^a$ denote the $a^{th}$ vector of the standard basis, $i.e.$ $e^1=\{1,0,0,0\},e^2=\{0,1,0,0\},e^3=\{0,0,1,0\}$ and $e^4=\{0,0,0,1\}$.
Substituting the coefficients $\xi^{\gamma,u}$,$\xi^{\gamma,v}$,$\xi^{\lambda,u}$,$\xi^{\lambda,v}$ and $\eta^{\gamma,u}$,$\eta^{\gamma,v}$,$\eta^{\lambda,u}$ ,$\eta^{\lambda,v}$ (eqs.~\eqref{xietacoefficients}) in eqs.~\eqref{derivativeuu}-\eqref{derivativevv}, we get
\begin{equation}\label{derivativeuu1}
\mathcal{P}_{uu}(u,v)=\,d \sum\limits_{\gamma\in\sigma^d}\xi^{\gamma,u}\left(\omega_\gamma p_\gamma, \omega_\gamma\right)
+d(d-1)\sum\limits_{\gamma\in\sigma^d}\eta^{\gamma,u}\left(\omega_\gamma p_\gamma, \omega_\gamma\right).
\end{equation}
\begin{equation}\label{derivativevv1}
\mathcal{P}_{vv}(u,v)=\,d \sum\limits_{\gamma\in\sigma^d}\xi^{\gamma,v}\left(\omega_\gamma p_\gamma, \omega_\gamma\right)
+d(d-1)\sum\limits_{\gamma\in\sigma^d}\eta^{\gamma,v}\left(\omega_\gamma p_\gamma, \omega_\gamma\right).
\end{equation}
so that the eqs.~\eqref{derivativeuucntrlpoint} and \eqref{derivativevvcntrlpoint} reduce to
\begin{equation}\label{derivativeuucntrlpoint1}
\frac{\partial \mathcal{P}_{uu}(u,v)}{\partial\left(\omega_\lambda p_\lambda, \omega_\lambda\right)^{a}}
=\,d \xi^{\lambda,u}e^a+d(d-1)\eta^{\lambda,u}e^a,
\end{equation}
and
\begin{equation}\label{derivativevvcntrlpoint1}
\frac{\partial \mathcal{P}_{vv}(u,v)}{\partial\left(\omega_\lambda p_\lambda, \omega_\lambda\right)^{a}}
=d \xi^{\lambda,v}e^a+d(d-1)\eta^{\lambda,v}e^a.
\end{equation}
Now substitute eqs.~\eqref{derivativeuu1} to~\eqref{derivativevvcntrlpoint1} in eq.~\eqref{functional} to get
\begin{equation}\label{functional3}
\begin{split}
& \frac{\partial \mathcal{H}(\mathcal{P}(u,v))}{\partial {{\left( {{\omega }_{\lambda }}{{p}_{\lambda }},{{\omega }_{\lambda }} \right)}^{a}}}=\, 2{{d}^{2}}\int\limits_{{{I}^{d}}}{\sum\limits_{\gamma \in {{\sigma }^{d}}}{\left( ({{\xi }^{\lambda ,u}}+(d-1){{\eta }^{\lambda ,u}})+({{\xi }^{\lambda ,v}}+(d-1){{\eta }^{\lambda ,v}}) \right)}} \\
 & \left( ({{\xi }^{\gamma ,u}}+(d-1){{\eta }^{\gamma ,u}})+({{\xi }^{\gamma ,v}}+(d-1){{\eta }^{\gamma ,v}}) \right)({{\omega }_{\gamma }}{{p}_{\gamma }},{{\omega }_{\gamma }})dudv.
\end{split}
\end{equation}
We can now obtain the set of linear system of equations as stated in eq. \eqref{finalfunctional} for which the toric B{\'e}zier patch is quasi-harmonic surface by setting $\frac{\partial\mathcal{H}(\mathcal{P}(u,v))}{\partial\left(\omega_\lambda p_\lambda,\omega_\lambda\right)^{a}}=0$.
\end{proof}
\begin{comment}The linear system of equations eq. \eqref{finalfunctional} is the constrained condition for the unknown inner mass-points of the toric B{\'e}zier patch with known boundary mass points to be quasi-harmonic. The following section \ref{examples} gives illustrations of how this linear system of equations is used to define quasi-harmonic toric B{\'e}zier patches when inner mass-points are unknown.
\end{comment}
\section{Quasi-Harmonic Toric B{\'e}zier patches over multi-sided Convex hulls}\label{examples}
In this section, we construct toric B{\'e}zier patches over two different convex hulls namely 1) the trapezoidal convex hull and  2) hexagonal convex hull.  We use the linear set of equations given in eq.~\eqref{finalfunctional} to compute the inner unknown mass-points of the toric B\'{e}zier patches spanned by the curves over these convex hulls in order to obtain the associated quasi-harmonic toric B\'{e}zier patch. In the former case we construct the quasi-harmonic toric B\'{e}zier patches for $n=2,m=p=1$ in which we find one condition on the unknown inner mass point and $n=2,m=3,p=1$, we find three conditions on the three unknown inner mass points whereas in the latter case we construct the quasi-harmonic toric B\'{e}zier patch with depth $d=2$, in this case there appear seven unknown inner points in terms of known boundary mass points. To simplify the calculations, the weights $\omega$ are all taken equal.
\subsection{Quasi-harmonic toric B{\'e}zier patches over trapezoidal convex hulls}\label{trep}
The general representation of toric B\'{e}zier patch $B(u,v)$ over a trapezoidal convex hull $I_{\sigma}$ is defined as follows,

Let $n$, $p\geq1$ and $m\geq0$ be integers and set
\begin{equation}\label{trapezoidal}
\sigma=\{(i,j): 0\leq j\leq n,~0\leq i \leq m+pn+pj\}
\end{equation}
be the collection of all the integers lattice points of the trapezoidal convex hull. The corresponding Bernstein polynomial for the trapezoidal convex hull is given as
\begin{equation}
\bm{\beta}_{ij}(u,v)=c_{ij} u^{i}(m+pn-pv-u)^{m+pn-pj-i}v^{i}(n-v)^{n-j}.
\end{equation}
Then the toric B\'{e}zier surface $B(u,v)$ defined over a general trapezoidal hull is expressed as
\begin{equation}
\mathcal{P}(u,v)=\sum\limits_{(i,j)\in I}c_{ij} u^{i}(m+pn-pv-u)^{m+pn-pj-i}v^{i}(n-v)^{n-j}\left(\omega_{ij} P_{ij},\omega_{ij}\right),~~~(u,v)\in I_\sigma.
\end{equation}
\begin{example}
In particular, for $n=2,m=p=1$, the eq.~\eqref{trapezoidal} gives us the following set of integer lattice-points
\begin{equation*}
\sigma=\{(0,0),(1,0),(2,0),(3,0),(0,1),(0,2),(1,2),(2,1),(1,1)\},
\end{equation*}
with only  one inner unknown mass-point $p_{11}$ associated to $\sigma_i=(1,1)$.
\begin{figure}[h!]\centering
  % Requires \usepackage{graphicx}
  \includegraphics[width=60mm]{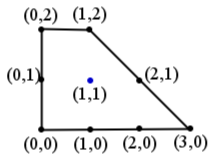}\\
  \caption{Trapezoidal domain with 1 inner lattice point,shown as the blue dot,
  indexing the corresponding unknown mass-point}\label{1.1}
\end{figure}
\noindent The Bernstein polynomials
\begin{equation}\label{beta}
\bm{\beta}_{ij}(u,v)=c_{ij} u^i (2-v)^{2-j} v^j (3-u-v)^{3-i-j},
\end{equation}
for the respective lattice-points come out to be
\begin{equation}
\begin{aligned}
\bm{\beta}_{00}(u,v)&=\frac{1}{108} (2-v)^2 (-u-v+3)^3, &
\bm{\beta}_{10}(u,v)&=\frac{1}{16} u (2-v)^2 (-u-v+3)^2,\\
\bm{\beta}_{20}(u,v)&=\frac{1}{16} u^2 (2-v)^2 (-u-v+3), &
\bm{\beta}_{30}(u,v)&=\frac{1}{108} u^3 (2-v)^2,\\
\bm{\beta}_{01}(u,v)&=\frac{1}{4} (2-v) v (-u-v+3)^2,&
\bm{\beta}_{02}(u,v)&=\frac{1}{4} v^2 (-u-v+3),\\
\bm{\beta}_{12}(u,v)&=\frac{u v^2}{4},&
\bm{\beta}_{21}(u,v)&=\frac{1}{4} u^2 (2-v) v,\\
\bm{\beta}_{11}(u,v)&=u (2-v) v (-u-v+3),
\end{aligned}
\end{equation}
in which  $c_{ij}$ have been chosen appropriately. The toric B\'{e}zier patch over the given trapezoidal convex hull $I_\sigma$, as shown in fig.\ref{1.1} with corresponding Bernstein polynomials defined over lattice points is expressed as
\begin{equation}
\mathcal{P}(u,v)=\sum\limits_{(i,j)\in I}\bm{\beta}_{ij}(u,v)(\omega_{ij} P_{ij},\omega_{ij}),
\end{equation}
where $(u,v)\in I_\sigma$.
We find the constraints for the toric B{\'e}zier patch with unknown inner mass-points to be quasi-harmonic by substituting the second order partial derivative and their gradient with respect to the inner unknown mass-point $p_{11}$ in eq.~\eqref{finalfunctional}. The toric B{\'e}zier patch is quasi-harmonic if and only if the mass-points of the patch satisfy the following constraint equation
\begin{equation}\label{1inner}
p_{11}=0.0904 p_{00}-0.1973 p_{01}+0.01430 p_{02}+0.1970 p_{10}+0.1006 p_{12}+ 0.09269 p_{20}-0.1390 p_{21}+ 0.0438 p_{30}.
\end{equation}
\end{example}
\begin{figure}[h!]\centering
  % Requires \usepackage{graphicx}
  \includegraphics[width=75mm]{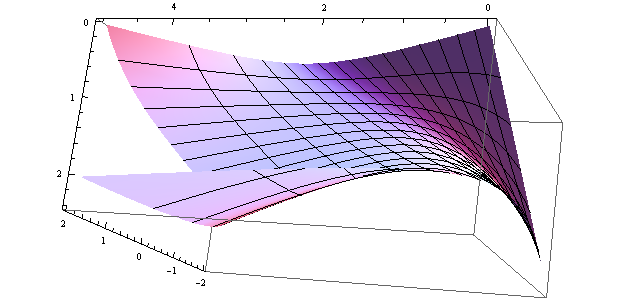}\\
  \caption{A quasi-harmonic toric B\'{e}zier patch with 1 inner lattice point
  indexing the unknown mass-point which is computed by using the eq.~\eqref{1inner}}\label{1.2}
\end{figure}
A particular example of a toric B{\'e}zier patch over trapezoidal convex hull  with 1 unknown inner mass-point is given in figure \ref{1.2} by taking known mass-points on the boundary of the convex hull. The unknown inner mass-point $p_{11}$ is computed by using the result as stated in eq.~\eqref{1inner}.
\begin{example}
For $n=2,m=3,p=1$,the set of integer lattice points is given as
\begin{equation*}
\sigma=\{(0,0),(1,0),(2,0),(3,0),(4,0),(5,0),(0,1),(0,2),(1,2),(2,2),(3,2),(2,1),(3,1),(4,1),(1,1)\}
\end{equation*}
with $3$ inner unknown mass-point $p_{11},p_{21}$ and $p_{31}$.
\begin{figure}[tbh!]\centering
  \includegraphics[width=60mm]{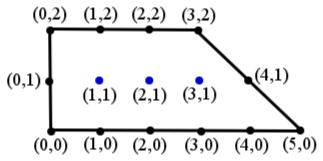}
  \captionof{figure}{Trapezoidal domain with 3 inner lattice point, marked as blue dots, associated to the corresponding unknown mass-points}
  \label{2.1}
\end{figure}
\noindent The Bernstein polynomials
\begin{equation}\label{beta1}
\bm{\beta}_{ij}(u,v)=c_{ij} u^i (2-v)^{2-j} v^j (5-u-v)^{5-i-j},
\end{equation}
for the respective lattice points are
\begin{equation}
\begin{aligned}
\bm{\beta}_{00}(u,v)&=\frac{1}{12500}(2-v)^2 (5-u-v)^5, &
\bm{\beta}_{10}(u,v)&=\frac{1}{2500}u (2-v)^2 (5-u-v)^4, \\
\bm{\beta}_{20}(u,v)&=\frac{1}{1250}u^2 (2-v)^2 (5-u-v)^3, &
\bm{\beta}_{30}(u,v)&=\frac{1}{1250}u^3 (2-v)^2 (5-u-v)^2, \\
\bm{\beta}_{40}(u,v)&=\frac{1}{2500}u^4 (2-v)^2 (5-u-v), &
\bm{\beta}_{50}(u,v)&=\frac{1}{12500}u^5 (2-v)^2, \\
\bm{\beta}_{01}(u,v)&=\frac{1}{512} (2-v) (5-u-v)^4 v, &
\bm{\beta}_{02}(u,v)&=\frac{1}{108} (5-u-v)^3 v^2, \\
\bm{\beta}_{11}(u,v)&=\frac{1}{128} u (2-v) (5-u-v)^3 v, &
\bm{\beta}_{21}(u,v)&=\frac{3}{256} u^2 (2-v) (5-u-v)^2 v, \\
\bm{\beta}_{31}(u,v)&=\frac{1}{128} u^3 (2-v) (5-u-v) v, &
\bm{\beta}_{41}(u,v)&=\frac{1}{512} u^4 (2-v) v, \\
\bm{\beta}_{12}(u,v)&=\frac{1}{36} u (5-u-v)^2 v^2, &
\bm{\beta}_{22}(u,v)&=\frac{1}{36} u^2 (5-u-v) v^2, \\
\bm{\beta}_{32}(u,v)&=\frac{1}{108}(u^3 v^2).
\end{aligned}
\end{equation}
for an appropriate choice of $c_{ij}$.
The toric B\'{e}zier patch over the given convex hull $I_\sigma$ is defined as
\begin{equation}
\mathcal{P}(u,v)=\sum\limits_{(i,j)\in I}\bm{\beta}_{ij}(u,v)(\omega_{ij} p_{ij},\omega_{ij}),
\end{equation}
where $(u,v)\in I_\sigma$.
We can find the constraints for the toric B{\'e}zier patch with unknown inner mass-points to be quasi-harmonic by substituting the second order partial derivatives and their gradient with respect to each unknown inner mass points, namely $p_{11},p_{21}$ and $p_{31}$ in eq.~\eqref{finalfunctional}. The toric B{\'e}zier patch over the given trapezoidal convex hull is quasi-harmonic if and only if the mass-points of the patch satisfy the following system of equations
\begin{equation}\label{3inner}
\begin{split}
p_{11}&=0.8157 p_{00}-1.027 p_{01}+0.3170 p_{02}+0.1815 p_{10}+0.2146 p_{12}+0.1106 p_{20}-0.0981 p_{22}+0.06520 p_{30}\\&+0.1122 p_{32}+0.0210 p_{40}-0.07081 p_{41}+0.02990 p_{50}, \\
p_{21}&=-0.4561 p_{00}+0.5450 p_{01}-0.1718 p_{02}+0.2450 p_{10}+0.0246 p_{12}+0.2527 p_{20}+0.7647 p_{22}+0.03338 p_{30}\\&-0.3465 p_{32}+0.02031 p_{40}+0.2384 p_{41}-0.09169 p_{50}, \\
p_{31}&=0.0871 p_{00}-0.0955 p_{01}+0.0341 p_{02}-0.0211 p_{10}-0.0221 p_{12}+0.1053 p_{20}-0.3099 p_{22}+0.3448 p_{30}\\&+0.6595 p_{32}+0.1629 p_{40}-0.4250 p_{41}+0.2278 p_{50}.
\end{split}
\end{equation}
\end{example}
\begin{figure}[h!]\centering
  \includegraphics[width=60mm]{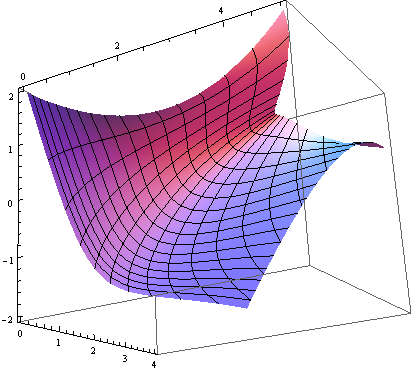}\\
  \caption{A quasi-harmonic toric patch defined over trapezoidal convex hull
   with 3 inner lattice points indexing the unknown mass-points which are computed by using system of eqs.~\eqref{3inner}}\label{2.2}
\end{figure}

\subsection{Quasi-harmonic toric B{\'e}zier patches of depth 2}
Consider a toric B{\'e}zier patch defined over hexagonal convex hull,  shown by the dotted line in fig.~\ref{hex} with lattice-points,
\begin{equation*}
\sigma=\{(0,0),(1,0),(0,1),(1,1),(1,2),(2,1),(2,2)\},
\end{equation*}
where the edges of the hexagonal convex hull $I_{\sigma}$  are
\begin{equation*}
\overline{L}_{1}(u,v)=v;~\overline{L}_{2}(u,v)=-v+2;~\overline{L}_{3}(u,v)=-u+2;~\overline{L}_{4}(u,v)=u;~\overline{L}_{5}(u,v)=v-u+1;~\overline{L}_{6}(u,v)=-v+u+1.
\end{equation*}
The toric Bernstein polynomials for each lattice-point $\sigma_{i}\in\sigma$ can be defined using the following relation
\begin{equation}\label{beta2}
\bm{\beta}_{\sigma_{i}}= c_{\sigma_{i}}\overline{L}_{1}(u,v)^{\overline{L}_{1}({\sigma_{i}})}\overline{L}_{2}(u,v)^{\overline{L}_{2}({\sigma_{i}})}\overline{L}_{3}(u,v)^{\overline{L}_{3}({\sigma_{i}})}\overline{L}_{4}(u,v)^{\overline{L}_{4}({\sigma_{i}})}\overline{L}_{5}(u,v)^{\overline{L}_{5}({\sigma_{i}})}\overline{L}_{6}(u,v)^{\overline{L}_{6}({\sigma_{i}})}.
\end{equation}
Whereas, the Bernstein polynomials $\{\bm{\beta}^d_{\gamma}\}_{\gamma\in I^{d}}$ for the toric B{\'e}zier patch of depth $d=2$ can be computed by convolving the Bernstein polynomials $\bm{\beta}_{\sigma_i}(u,v)$ as stated above in eq.~\eqref{beta2} indexed by the Minkowski sum $\sigma\oplus\sigma=\sigma^{2}$.
The toric B\'{e}zier patch of depth $2$ over the hexagonal convex hull (as shown as solid line in  fig.~\ref{hex}) $I^d$, with corresponding Bernstein polynomials is defined as
\begin{equation}
\mathcal{P}(u,v)=\sum\limits_{\gamma\in \sigma^d} \bm{\beta}^{d}_{\gamma}(u,v) \left(\omega_\gamma p_\gamma, \omega_\gamma\right),
\end{equation}
where $\left(u,v\right)\in I^d$. Similarly, as we already have shown for  the toric B\'{e}zier patches over trapezoidal convex hull, the constraints on the mass-points for this patch can also be computed by using eq.~\eqref{finalfunctional}. The toric B{\'e}zier patch of depth $2$ over the hexagonal convex hull with $7$ unknown inner-mass points is quasi-harmonic if and only if these inner- mass points of the patch satisfy the following linear system of constraints
\begin{figure}[tbh!]
  % Requires \usepackage{graphicx}
  \includegraphics[width=60mm]{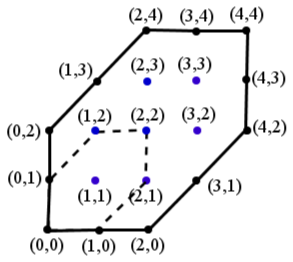}
  \caption{A hexagonal convex hull (solid line) of depth $d=2$ with 19 lattice points indexed by the set $I_{\sigma^2}$ with 7 inner lattice points, marked as blue dots corresponding to the unknown mass-points. The dotted lines represent  the hexagonal hull of $I_{\sigma}$ for toric B\'{e}zier patch of depth $d=1$}\label{hex}
\end{figure}
\begin{comment}
\begin{figure}[tbh!]
  % Requires \usepackage{graphicx}
  \includegraphics[width=80mm]{hexd1}
  \caption{Quasi-harmonic toric patch defined over hexagonal hull of depth $d=1$ }\label{hexd1}
\end{figure}
\end{comment}
\begin{equation}
\begin{split}
p_{11}&=-0.1986 p_{00}-0.0264 p_{01}-0.8067 p_{02}+0.1294 p_{10}-5.9296 p_{13}-0.4091 p_{20}-19.2105 p_{24}-5.8759p_{31}\\&-38.9767 p_{34}-41.2599 p_{42} +23.8580 p_{44},\\
p_{21}&=-0.0018 p_{00}+0.0033 p_{01}-0.0582 p_{02}+0.0045 p_{10}-1.9341p_{13}-0.0698 p_{20}-0.9901 p_{24}-0.2803 p_{31}\\&+10.1815 p_{34}-1.1566 p_{42}+2.804 p_{44},\\
p_{12}&=0.1143 p_{00}-0.5075 p_{01}+5.7024 p_{02}-0.3095 p_{10}+25.0159 p_{13}+6.2165 p_{20}+110.1030 p_{24}+25.0099 p_{31}\\&-3.9603 p_{34}+104.0531 p_{42}-142.3752 p_{44},\\
p_{22}&=0.0004 p_{01}-0.0434 p_{02}+0.0034 p_{10}-0.1725 p_{13}-0.0285 p_{20}-0.3504 p_{24}-0.1228 p_{31}-0.41809p_{34}\\&-0.4812 p_{42}-0.0297 p_{44},\\
p_{32}&=-0.0010 p_{00}+0.0020 p_{01}+0.4605 p_{02}-0.0834 p_{10}+2.6567 p_{13}+0.1535 p_{20}+8.5203 p_{24}+1.1316 p_{31}\\&-0.0712 p_{34}+6.7355p_{42}+4.4061 p_{44},\\
p_{23}&=-0.0168 p_{00}+0.0555 p_{01}-0.9415 p_{02}+0.0297 p_{10}-4.0360 p_{13}-0.8251 p_{20}-19.3765 p_{24}-3.0281 p_{31}\\&-2.3478 p_{34}-11.6685 p_{42}+10.8462 p_{44},\\
p_{33}&=0.0005 p_{00}-0.0007 p_{01}-0.0393 p_{02}+0.0060 p_{10}-0.2230 p_{13}-0.0299 p_{20}-0.9192 p_{24}-0.1673 p_{31}\\&+0.2163 p_{34}-0.9267 p_{42}-1.0000 p_{43}-2.9665p_{44}.
\end{split}
\end{equation}
Toric B{\'e}zier patches defined over any polygonal convex hull of domains with prescribed boundary mass-points can be approximated to quasi-harmonic toric B{\'e}zier patch using the result stated in eq.~\eqref{finalfunctional}.
\section{Conclusion}
In this paper, we considered the quasi-Plateau problem which consists of finding the quasi-minimal surface with prescribed entire or partial border. In particular, we find a solution to the \emph{Plateau-toric B{\'e}zier problem} for toric B\'{e}zier surface, which is the generalization of classical rational triangular and tensor-product B{\'e}zier surfaces defined over multi-sided domains. Quasi harmonic functional is used as the objective functional which is extremized to obtain a toric B\'{e}zier patch among all the possible patches with prescribed boundary, which we termed as \emph{quasi-harmonic toric B{\'e}zier patch}. This patch serves as the solution to \emph{quasi Plateau-toric B{\'e}zier problem}. The vanishing condition for gradient of the quasi-harmonic functional yields the constraints on mass-points of the toric B\'{e}zier patch as system of linear equations under which it is a quasi-harmonic toric B\'{e}zier patch. This scheme is applied to toric B\'{e}zier patches for different prescribed borders defined over the multi-sided convex hulls to illustrate its effectiveness and flexibility.
\newpage
\bibliographystyle{unsrt}
\bibliography{ToricBezierSurface}

\begin{thebibliography}{10}

\bibitem{Osserman1986}
R.~Osserman.
\newblock {\em A survey of Minimal Surfaces}.
\newblock Dover Publications Inc., 1986.

\bibitem{Nitsche1989}
J.~C.~C. Nitsche.
\newblock {\em Lectures on Minimal Surfaces}.
\newblock Cambridge University Press, 1989.

\bibitem{Plateau}
J.A.F Plateau.
\newblock Statique expérimentale et théorique des liquides soumis aux seules
  forces moléculaires.
\newblock {\em Gauthier-Villars, Paris}, 1873.

\bibitem{Schwarz}
H.A. Schwarz.
\newblock \textsc{G}esammelte \textsc{M}athematische \textsc{A}bhandlungen.
\newblock {\em 2 B\text{\" a}nde. Springer}, 1890.

\bibitem{Garnier}
R.~Garnier.
\newblock Le problème de \textsc{P}lateau.
\newblock {\em Annales \textsc{S}cientifiques de l'E.N.S.},
  \textbf{3(45)}:53--144, 1928.

\bibitem{Douglas}
J.~Douglas.
\newblock Solution of the problem of \textsc{P}lateau.
\newblock {\em Trans. Amer. Math. Soc.}, \textbf{33(1)}:263--321, 1931.

\bibitem{Rado}
T.~Rad\'o.
\newblock On \textsc{P}lateau's problem.
\newblock {\em Ann. Of Math.}, \textbf{(2)31(3)}:457--469, 1930.

\bibitem{COPPIN1988315}
C.~Coppin and D.~Greenspan.
\newblock A contribution to the particle modeling of soap films.
\newblock {\em Applied Mathematics and Computation}, 26(4):315 -- 331, 1988.

\bibitem{Koohestani20142071}
K.~Koohestani.
\newblock Nonlinear force density method for the form-finding of minimal
  surface membrane structures.
\newblock {\em Communications in Nonlinear Science and Numerical Simulation},
  19(6):2071 -- 2087, 2014.

\bibitem{brakke1992}
Kenneth~A. Brakke.
\newblock The surface evolver.
\newblock {\em Experiment. Math.}, 1(2):141--165, 1992.

\bibitem{Chopp199377}
D.~L. Chopp.
\newblock Computing minimal surfaces via level set curvature flow.
\newblock {\em Journal of Computational Physics}, 106(1):77 -- 91, 1993.

\bibitem{Jung:2007}
Y.Jung, K.T.Chu, and S.Torquato.
\newblock A variational level set approach for surface area minimization of
  triply-periodic surfaces.
\newblock {\em J. Comput. Phys.}, 223(2):711--730, May 2007.

\bibitem{Trasdahl20114795}
{\O}.~Tr{\aa}sdahl and E.M. R{\o}nquist.
\newblock High order numerical approximation of minimal surfaces.
\newblock {\em Journal of Computational Physics}, 230(12):4795 -- 4810, 2011.

\bibitem{Li20136415}
C.~Y. Li, R.~H. Wang, and C.~G. Zhu.
\newblock Designing approximation minimal parametric surfaces with geodesics.
\newblock {\em Applied Mathematical Modelling}, 37(9):6415 -- 6424, 2013.

\bibitem{Kassabov2014441}
O.~Kassabov.
\newblock Transition to canonical principal parameters on minimal surfaces.
\newblock {\em Computer Aided Geometric Design}, 31(7–8):441 -- 450, 2014.

\bibitem{Xu:2015:EFP:2778896.2779237}
Gang Xu, Yaguang Zhu, Guozhao Wang, Andr{\'e} Galligo, Li~Zhang, and Kin-chuen
  Hui.
\newblock Explicit form of parametric polynomial minimal surfaces with
  arbitrary degree.
\newblock {\em Appl. Math. Comput.}, 259(C):124--131, May 2015.

\bibitem{Monterde2004}
J.~Monterde.
\newblock \textsc{B}\'{e}zier surfaces of minimal area: The \textsc{D}irichlet
  approach.
\newblock {\em Computer Aided Geometric Design}, \textbf{21}:117--136, 2004.

\bibitem{Chen2009}
X.~D. Chen, G.~Xu, and Y.~Wang.
\newblock Approximation methods for the \textsc{P}lateau-\textsc{B}\'{e}zier
  problem.
\newblock In {\em 2009 11th IEEE International Conference on Computer-Aided
  Design and Computer Graphics, 2009}.

\bibitem{Hao2012}
Y.~X. Hao, R.~H. Wang, and C.~J. Li.
\newblock Minimal quasi-\textsc{B}\'{e}zier surface.
\newblock {\em Applied Mathematical Modelling}, \textbf{36}:5751 -- 5757, 2012.

\bibitem{Xu2010}
G.~Xu and G.~Wang.
\newblock Quintic parametric polynomial minimal surfaces and their properties.
\newblock {\em Differential Geometry and its Applications}, \textbf{28}:697 --
  704, 2010.

\bibitem{Ahmad201472}
D.~Ahmad and B.~Masud.
\newblock Variational minimization on string-rearrangement surfaces,
  illustrated by an analysis of the bilinear interpolation.
\newblock {\em Applied Mathematics and Computation}, 233:72 -- 84, 2014.

\bibitem{dabm2013}
D.~Ahmad and B.~Masud.
\newblock A \textsc{C}oons patch spanning a finite number of curves tested for
  variationally minimizing its area.
\newblock {\em Abstract and Applied Analysis}, 2013, 2013.

\bibitem{Ahmad20151242}
D.~Ahmad and B.~Masud.
\newblock Near-stability of a quasi-minimal surface indicated through a tested
  curvature algorithm.
\newblock {\em Computers \& Mathematics with Applications}, 69(10):1242 --
  1262, 2015.

\bibitem{Xu2015}
G.~Xu, T.~Rabczuk, E.~G\"{u}ler, Q.~Wu, K.~C. Hui, and G.~Wang.
\newblock Quasi-harmonic \textsc{B}{\'e}zier approximation of minimal surfaces
  for finding forms of structural membranes.
\newblock {\em Comput. Struct.}, 161(C):55--63, December 2015.

\bibitem{Chern1955}
S.~S. Chern.
\newblock {\em An elementary proof of the existence of isothermal parameters on
  a surface}.
\newblock Proc. Amer. Math. Soc., 60: 771-782,1955.

\bibitem{MonterdeUgail2004}
J.~Monterde and H.~Ugail.
\newblock On harmonic and biharmonic \textsc{B}\'{e}zier surfaces.
\newblock {\em Computer Aided Geometric Design}, \textbf{21}:697 -- 715, 2004.

\bibitem{FARIN198683}
G.~Farin.
\newblock Triangular \textsc{B}ernstein-\textsc{B}\'{e}zier patches.
\newblock {\em Computer Aided Geometric Design}, 3(2):83 -- 127, 1986.

\bibitem{Warren1994}
J.~Warren.
\newblock {\em Algebraic Geometry and its Applications}, chapter A Bound on the
  Implicit Degree of Polygonal \textsc{B}{\'e}zier Surfaces, pages 513--525.
\newblock Springer New York, 1994.

\bibitem{ref1}
R.~Krasauskas.
\newblock Toric surface patches.
\newblock {\em Advances in Computational Mathematics}, 17(1):89--113, 2002.

\bibitem{KrausGoldman}
R.~Goldman, R. \and~Krasauskas.
\newblock {\em Topics in Algebraic Geometry and Geometric Modeling}, chapter
  Toric \textsc{B}{\'e}zier patches with depth.
\newblock American Mathematical Society, 2002.

\bibitem{XU20171}
Gang Xu, Tsz-Ho Kwok, and Charlie~C.L. Wang.
\newblock Isogeometric computation reuse method for complex objects with
  topology-consistent volumetric parameterization.
\newblock {\em Computer-Aided Design}, 91:1 -- 13, 2017.

\bibitem{Puente:2011}
L.~D. Garcia-Puente, F.~Sottile, and C.~Zhu.
\newblock Toric degenerations of \textsc{B}{\'{e}}zier patches.
\newblock {\em ACM Trans. Graph.}, 30(5):1--10, 2011.

\bibitem{Sun2015255}
L.~Y. Sun and C.~G. Zhu.
\newblock $\textsc{G}^{1}$ continuity between toric surface patches.
\newblock {\em Computer Aided Geometric Design}, 35–-36:255 -- 267, 2015.

\bibitem{Sun01062016}
L.Y. Sun and C.G. Zhu.
\newblock Approximation of minimal toric \textsc{B}\'{e}zier patch.
\newblock {\em Advances in Mechanical Engineering}, 8(6), 2016.

\bibitem{Monterde2006}
J.~Monterde and H.~Ugail.
\newblock A general 4th-order \textsc{PDE} method to generate
  \textsc{B}\'{e}zier surfaces from the boundary.
\newblock {\em Computer Aided Geometric Design}, \textbf{23}:208 -- 225, 2006.

\bibitem{farin}
G.~Farin.
\newblock {\em Curves and Surfaces for Computer Aided Geometric Design}.
\newblock The Academic Press, USA, 2002.

\bibitem{Schneider2001359}
R.~Schneider and L.~Kobbelt.
\newblock Geometric fairing of irregular meshes for free-form surface design.
\newblock {\em Computer Aided Geometric Design}, 18(4):359 -- 379, 2001.

\bibitem{Bloor2005203}
M.I.G. Bloor and M.J. Wilson.
\newblock An analytic pseudo-spectral method to generate a regular 4-sided
  \textsc{PDE} surface patch.
\newblock {\em Computer Aided Geometric Design}, 22(3):203 -- 219, 2005.

\end{thebibliography}
\end{document}